\documentclass[11pt]{amsart}%
\usepackage{eurosym,amscd,amsmath,amsfonts,amssymb,amsthm,color,graphicx,latexsym,geometry,hyperref}
\usepackage{amsmath}
\usepackage{amsfonts}
\usepackage{amssymb}
\usepackage{graphicx}
\usepackage{tikz}
\usepackage{float}%
\providecommand{\U}[1]{\protect\rule{.1in}{.1in}}
\providecommand{\U}[1]{\protect\rule{.1in}{.1in}}
\providecommand{\U}[1]{\protect\rule{.1in}{.1in}}
\providecommand{\U}[1]{\protect\rule{.1in}{.1in}}
\newtheorem{theorem}{Theorem}[section]
\newtheorem{corollary}[theorem]{Corollary}
\newtheorem{proposition}[theorem]{Proposition}
\newtheorem{lemma}[theorem]{Lemma}

\theoremstyle{definition}

\newtheorem{remark}[theorem]{Remark}

\begin{document}
\title[Super-critical Hardy--Littlewood inequalities for multilinear forms]{Super-critical Hardy--Littlewood inequalities for multilinear forms}
\author[D. N\'u\~nez]{D. N\'{u}\~{n}ez--Alarc\'{o}n}
\address{Departamento de Matem\'{a}ticas\\
\indent Universidad Nacional de Colombia\\
\indent111321 - Bogot\'a, Colombia}
\email{danielnunezal@gmail.com and dnuneza@unal.edu.co}
\author[D. Paulino]{D. Paulino}
\address{Departamento de Matem\'{a}tica \\
Universidade Federal da Para\'{\i}ba \\
58.051-900 - Jo\~{a}o Pessoa, Brazil.}
\email{djair.paulino@ifrn.edu.br and djairpsc@hotmail.com}
\author[D. Pellegrino]{D. Pellegrino*}
\address{Departamento de Matem\'{a}tica \\
Universidade Federal da Para\'{\i}ba \\
58.051-900 - Jo\~{a}o Pessoa, Brazil.}
\email{pellegrino@pq.cnpq.br and dmpellegrino@gmail.com}
\thanks{*Supported CNPq Grant 307327/2017-5 and by Grant 2019/0014 Para\'{\i}ba State
Research Foundation (FAPESQ)}

\begin{abstract}
The multilinear Hardy--Littlewood inequalities provide estimates for the sum
of the coefficients of multilinear forms $T:\ell_{p_{1}}^{n}\times\cdots
\times\ell_{p_{m}}^{n}\rightarrow\mathbb{R}$ (or $\mathbb{C}$) when
$1/p_{1}+\cdots+1/p_{m}<1.$ In this paper we investigate the critical and
super-critical cases; i.e., when $1/p_{1}+\cdots+1/p_{m}\geq1.$

\end{abstract}
\maketitle

\section{Introduction}

Littlewood's $4/3$ theorem assures that for $\mathbb{K}=\mathbb{R}$ or
$\mathbb{C}$, we have
\[
\left(  \sum_{j_{1}.j_{2}=1}^{n}\left\vert A(e_{j_{1}},e_{j_{2}})\right\vert
^{4/3}\right)  ^{3/4}\leq\sqrt{2}\left\Vert A\right\Vert
\]
for all positive integers $n$ and all bilinear forms $A:\ell_{\infty}%
^{n}\times\ell_{\infty}^{n}\rightarrow\mathbb{K}$, where as usual%
\[
\left\Vert A\right\Vert =\sup\left\{  \left\vert A(x,y)\right\vert :\left\Vert
x\right\Vert \leq1\text{ and }\left\Vert y\right\Vert \leq1 \right\}
\]
and $\ell_{p}^{n}$ denotes $\mathbb{K}^{n}$ with the $\ell_{p}$ norm; the
exponent $4/3$ cannot be improved (i.e., cannot be replaced by a smaller one).
Under an anisotropic viewpoint, the result can be generalized as follows (see
Theorem 5.1 in Pellegrino et al. 2017): the inequality
\begin{equation}
\left(  \sum_{j_{1}=1}^{n}\left(  \sum_{j_{2}=1}^{n}\left\vert A(e_{j_{1}%
},e_{j_{2}})\right\vert ^{a}\right)  ^{\frac{b}{a}}\right)  ^{\frac{1}{b}}%
\leq\sqrt{2}\left\Vert A\right\Vert \label{77}%
\end{equation}
holds for all $n$ whenever $a,b\in\lbrack1,\infty)$ satisfy
\[
\frac{1}{a}+\frac{1}{b}\leq\frac{3}{2}.
\]
Moreover, if $a,b\in\lbrack1,\infty)$ satisfy%
\[
\frac{1}{a}+\frac{1}{b}>\frac{3}{2},
\]
then (\ref{77}) is not possible, i.e., if%
\[
\left(  \sum_{j_{1}=1}^{n}\left(  \sum_{j_{2}=1}^{n}\left\vert A(e_{j_{1}%
},e_{j_{2}})\right\vert ^{a}\right)  ^{\frac{b}{a}}\right)  ^{\frac{1}{b}}\leq
C\left\Vert A\right\Vert ,
\]
then the constant $C$ must depend on $n$.

From now on, unless stated otherwise, the exponents involved in the
inequalities are positive and can be even infinity (in this case the
corresponding sum is replaced by the supremum). We also consider
$1/\infty:=0.$ The Hardy--Littlewood inequalities for bilinear forms were
conceived in 1934 by Hardy and Littlewood (see Theorem 5 in Hardy \&
Littlewood 1934), as a natural generalization of Littlewood's $4/3$
inequality. The results of the seminal paper of Hardy and Littlewood, in a
modern and somewhat more general presentation, can be summarized by the
following two theorems:

\begin{theorem}
\label{t1} (See Osikiewicz $\&$ Tonge 2001 and Aron et al. 2017) Let
$1<q\leq2<p$, with $\frac{1}{p}+\frac{1}{q}<1$. The following assertions are equivalent:

(a) There is a constant $C\geq1$ (not depending on $n$) such that%
\[
\left(  \sum_{j_{1}=1}^{n}\left(  \sum_{j_{2}=1}^{n}\left\vert A(e_{j_{1}%
},e_{j_{2}})\right\vert ^{a}\right)  ^{\frac{b}{a}}\right)  ^{\frac{1}{b}}\leq
C\left\Vert A\right\Vert
\]
for all bilinear forms $A:\ell_{p}^{n}\times\ell_{q}^{n}\rightarrow\mathbb{K}$
and all positive integers $n$.

(b) The exponents $a,b$ satisfy
\[
\left(  a,b\right)  \in\left[  \frac{q}{q-1},\infty\right)  \times\left[
\frac{1}{1-\left(  \frac{1}{p}+\frac{1}{q}\right)  },\infty\right)  .
\]
Moreover, the optimal constant $C$ is $1.$
\end{theorem}

\begin{theorem}
\label{psstbil}(see Pellegrino et al. 2017) Let $p,q\in\left[  2,\infty
\right]  $, with $\frac{1}{p}+\frac{1}{q}<1$. The following assertions are equivalent:

(a) There is a constant $C\geq1$ (not depending on $n$) such that%
\[
\left(  \sum_{j_{1}=1}^{n}\left(  \sum_{j_{2}=1}^{n}\left\vert A(e_{j_{1}%
},e_{j_{2}})\right\vert ^{a}\right)  ^{\frac{b}{a}}\right)  ^{\frac{1}{b}}\leq
C\left\Vert A\right\Vert
\]
for all bilinear forms $A:\ell_{p}^{n}\times\ell_{q}^{n}\rightarrow\mathbb{K}$
and all positive integers $n$.

(b) The exponents $a,b$ satisfy
\[
\left(  a,b\right)  \in\left[  \frac{q}{q-1},\infty\right)  \times\left[
\frac{1}{1-\left(  \frac{1}{p}+\frac{1}{q}\right)  },\infty\right)
\]
and
\begin{equation}
\frac{1}{a}+\frac{1}{b}\leq\frac{3}{2}-\left(  \frac{1}{p}+\frac{1}{q}\right)
. \label{78}%
\end{equation}

\end{theorem}

Since (\ref{78}) is trivially verified under the conditions of Theorem
\ref{t1}, we can unify the two theorems as follows:

\begin{theorem}
\label{azx}Let $q\in(1,\infty]$ and $p\in[2,\infty]$, with $\frac{1}{p}%
+\frac{1}{q}<1$. The following assertions are equivalent:

(a) There is a constant $C\geq1$ (not depending on $n$) such that%
\[
\left(  \sum_{j_{1}=1}^{n}\left(  \sum_{j_{2}=1}^{n}\left\vert A(e_{j_{1}%
},e_{j_{2}})\right\vert ^{a}\right)  ^{\frac{b}{a}}\right)  ^{\frac{1}{b}}\leq
C\left\Vert A\right\Vert
\]
for all bilinear forms $A:\ell_{p}^{n}\times\ell_{q}^{n}\rightarrow\mathbb{K}$
and all positive integers $n$.

(b) The exponents $a,b$ satisfy
\[
\left(  a,b\right)  \in\left[  \frac{q}{q-1},\infty\right)  \times\left[
\frac{1}{1-\left(  \frac{1}{p}+\frac{1}{q}\right)  },\infty\right)
\]
and%
\[
\frac{1}{a}+\frac{1}{b}\leq\frac{3}{2}-\left(  \frac{1}{p}+\frac{1}{q}\right)
.
\]

\end{theorem}

In 1981, Praciano--Pereira (see Praciano--Pereira 1981) extended the Hardy--Littlewood inequalities to
$m$-linear forms as follows: if $p_{1},...,p_{m}\in\lbrack1,\infty]$ and
\[
\frac{1}{p_{1}}+\cdots+\frac{1}{p_{m}}\leq\frac{1}{2},
\]
there exists a constant $C\geq1$ (not depending on $n$) such that
\begin{equation}
\left(  \sum_{j_{1},...,j_{m}=1}^{n}|T(e_{j_{1}},...,e_{j_{m}})|^{\frac
{2m}{m+1-2\left(  \frac{1}{p_{1}}+\cdots+\frac{1}{p_{m}}\right)  }}\right)
^{\frac{m+1-2\left(  \frac{1}{p_{1}}+\cdots+\frac{1}{p_{m}}\right)  }{2m}}\leq
C\Vert T\Vert, \label{bv}%
\end{equation}
for all $m$-linear forms $T:\ell_{p_{1}}^{n}\times\cdots\times\ell_{p_{m}}%
^{n}\rightarrow\mathbb{K}$ and for all positive integers $n$.

When
\[
\frac{1}{2}\leq\frac{1}{p_{1}}+\cdots+\frac{1}{p_{m}}<1,
\]
Dimant and Sevilla--Peris (see Dimant \& Sevilla--Peris 2016 and Cavalcante 2018) have proved that
there exists a constant $C\geq1$ (not depending on $n$) such that
\begin{equation}
\left(  \sum_{j_{1},...,j_{m}=1}^{n}|T(e_{j_{1}},...,e_{j_{m}})|^{\frac
{1}{1-\left(  \frac{1}{p_{1}}+\cdots+\frac{1}{p_{m}}\right)  }}\right)
^{1-\left(  \frac{1}{p_{1}}+\cdots+\frac{1}{p_{m}}\right)  }\leq C\Vert
T\Vert, \label{bvf}%
\end{equation}
for all $m$-linear forms $T:\ell_{p_{1}}^{n}\times\cdots\times\ell_{p_{m}}%
^{n}\rightarrow\mathbb{K}$ and for all positive integers $n$.

Both in (\ref{bv}) and (\ref{bvf}) the exponents are sharp, but there still
remains the question: what about anisotropic versions of (\ref{bv}) and
(\ref{bvf})?

In Albuquerque et al. 2014, the anisotropic version of the result of
Praciano--Pereira was finally settled (see also Santos \& Velanga 2017 for a
completer version for the case $p_{1},...,p_{m}=\infty$):

\begin{theorem}
\label{abps}(see Theorem 1.2 in Albuquerque et al. 2014 and Theorem 5.2 in
Pellegrino et al. 2017) Let $p_{1},...,p_{m}\in\lbrack1,\infty]$ be such that
\[
\frac{1}{p_{1}}+\cdots+\frac{1}{p_{m}}\leq\frac{1}{2}%
\]
and%
\[
q_{1},...,q_{m}\in\left[  \frac{1}{1-\left(  \frac{1}{p_{1}}+\cdots+\frac
{1}{p_{m}}\right)  },2\right]  .
\]
The following assertions are equivalent:

(a) There is a constant $C\geq1$ (not depending on $n$) such that
\[
\left(  \sum_{j_{1}=1}^{n}\left(  \cdots\left(  \sum_{j_{m}=1}^{n}\left\vert
A\left(  e_{j_{1}},\dots,e_{j_{m}}\right)  \right\vert ^{q_{m}}\right)
^{\frac{q_{m-1}}{q_{m}}}\cdots\right)  ^{\frac{q_{1}}{q_{2}}}\right)
^{\frac{1}{q_{1}}}\leq C\left\Vert A\right\Vert ,
\]
for all $m$-linear forms $A:\ell_{p_{1}}^{n}\times\cdots\times\ell_{p_{m}}%
^{n}\longrightarrow\mathbb{K}$ and all positive integers $n$.

(b) The inequality
\[
\frac{1}{q_{1}}+\cdots+\frac{1}{q_{m}}\leq\dfrac{m+1}{2}-\left(  \frac
{1}{p_{1}}+\cdots+\frac{1}{p_{m}}\right)
\]
is verified.
\end{theorem}

The anisotropic version of (\ref{bvf}) is still not completely solved, but in
Aron et al. 2017 the following partial answer (that also generalizes Theorem
\ref{t1}) was obtained:

\begin{theorem}
\label{aron} (see Theorem 3.2 in Aron et al. 2017) Let $m\geq2$ and
$1<p_{m}\leq2<p_{1},...,p_{m-1}$, with
\[
\dfrac{1}{p_{1}}+\cdots+\dfrac{1}{p_{m}}<1.
\]
The following assertions are equivalent:

(a) There is a constant $C\geq1$ (not depending on $n$) such that
\[
\left(  \sum_{j_{1}=1}^{n}\left(  \cdots\left(  \sum_{j_{m}=1}^{n}\left\vert
A\left(  e_{j_{1}},\dots,e_{j_{m}}\right)  \right\vert ^{q_{m}}\right)
^{\frac{q_{m-1}}{q_{m}}}\cdots\right)  ^{\frac{q_{1}}{q_{2}}}\right)
^{\frac{1}{q_{1}}}\leq C\left\Vert A\right\Vert ,
\]
for all $m$-linear forms $A:\ell_{p_{1}}^{n}\times\cdots\times\ell_{p_{m}}%
^{n}\longrightarrow\mathbb{K}$ and all positive integers $n.$

(b) The exponents $q_{1},...,q_{m}$ satisfy
\[
q_{1}\geq\delta_{m}^{p_{1},...,p_{m}},q_{2}\geq\delta_{m-1}^{p_{2},...,p_{m}%
},...,q_{m-1}\geq\delta_{2}^{p_{m-1},p_{m}},q_{m}\geq\delta_{1}^{p_{m}},
\]
with
\[
\delta_{m-k+1}^{p_{k},...,p_{m}}:=\dfrac{1}{1-\left(  \frac{1}{p_{k}}%
+\cdots+\frac{1}{p_{m}}\right)  }.
\]

\end{theorem}

The attentive reader may wonder why the case
\begin{equation}
\dfrac{1}{p_{1}}+\cdots+\dfrac{1}{p_{m}}\geq1 \label{3w}%
\end{equation}
is not investigated in the previous results? The reason is simple, because in
this case it is easy to prove that if there exists $C$ (not depending in $n$)
such that
\[
\left(  \sum_{j_{1},...,j_{m}=1}^{n}\left\vert T\left(  e_{j_{1}}%
,\dots,e_{j_{m}}\right)  \right\vert ^{s}\right)  ^{\frac{1}{s}}\leq
C\left\Vert T\right\Vert , \label{sigualinfinito}%
\]
for all $m$-linear forms $T:\ell_{p_{1}}^{n}\times\cdots\times\ell_{p_{m}}%
^{n}\longrightarrow\mathbb{K}$ and all positive integers $n$, then $s=\infty$
(i.e., we are forced to deal with the $\sup$ norm, and the result becomes
trivial). However, under the anisotropic viewpoint, as a matter of fact, there
is no reason to avoid the case (\ref{3w}) and it constitutes a vast field yet
to be explored. The first step in this direction is the following:

\begin{theorem}
\label{criticodjair}(see Theorem 1 in Paulino 2019) For all $m\geq2$ we have
\begin{equation}
\sup_{j_{1}}\left(  \sum_{j_{2}=1}^{n}\left(  \cdots\left(  \sum_{j_{m}=1}%
^{n}\left\vert T\left(  e_{j_{1}},\dots,e_{j_{m}}\right)  \right\vert ^{q_{m}%
}\right)  ^{\frac{q_{m-1}}{q_{m}}}\cdots\right)  ^{\frac{q_{2}}{q_{3}}%
}\right)  ^{\frac{1}{q_{2}}}\leq2^{\frac{m-2}{2}}\left\Vert T\right\Vert
\label{91}%
\end{equation}
for all $m$-linear forms $T:\ell_{m}^{n}\times\cdots\times\ell_{m}%
^{n}\rightarrow\mathbb{K},$ and all positive integers $n$, with
\[
q_{k}=\frac{2m(m-1)}{mk-2k+2}%
\]
for all $k=2,....,m.$ Moreover, $q_{1}=\infty$ and $q_{2}=m$ are sharp and,
for $m>2$ the optimal exponents $q_{k}$ satisfying (\ref{91}) fulfill
\[
q_{k}\geq\frac{m}{k-1}, k=2,....,m.
\]

\end{theorem}

The case considered in Theorem \ref{criticodjair} is called critical because
it is a special case of (\ref{3w}), and from now on we shall call the case
(\ref{3w}) as super-critical case, which is the topic of the present paper. In
Section 2 we provide a partial solution to the super-critical case for
$3$-linear forms and in Section 3 we investigate what are the conditions
needed to obtain $m$-linear Hardy--Littlewood inequalities in the
super-critical case.

\section{The $3$-linear case}

We begin this section by presenting two simple, albeit very useful, lemmas
that will be used all along the paper.

\subsection{Two multi-purpose lemmas}

For $S=\{s_{1},\dots,s_{k}\}\subset\{1,\dots,m\}$, we define
\[
\widehat{S}:=\{1,\dots,m\}\setminus S
\]
and by $\mathbf{i}_{S}$ we shall mean $(i_{s_{1}},\dots,i_{s_{k}})$. If
$S=\{s_{1},\dots,s_{k}\}$ and $\mathbf{p}=\left(  p_{1},...,p_{m}\right)
\in(0,\infty]^{m}$, we define
\[
\left\vert \frac{1}{\mathbf{p}}\right\vert _{S}:=\frac{1}{p_{s_{1}}}%
+\cdots+\frac{1}{p_{s_{k}}}.
\]
The lemmas read as follows:

\begin{lemma}
\label{t2}Let $k\in\left\{  1,...,m\right\}  $ and $\mathbf{p=}\left(
p_{1},...,p_{m}\right)  \in\lbrack1,\infty]^{m}$. Let $S=\{s_{1},\dots
,s_{k}\}\subset\{1,\dots,m\}$. If there is a constant $C\geq1$ such that%
\[
\left(  \sum_{j_{s_{1}}=1}^{n}\left(  \sum_{j_{_{s_{2}}}=1}^{n}\cdots\left(
\sum_{j_{_{s_{k}}}=1}^{n}\left\vert A(e_{j_{s_{1}}},...,e_{j_{s_{k}}%
})\right\vert ^{q_{k}}\right)  ^{\frac{q_{k-1}}{q_{k}}}\cdots\right)
^{\frac{q_{1}}{q_{2}}}\right)  ^{\frac{1}{q_{1}}}\leq C\left\Vert
A\right\Vert
\]
for all $k$-linear forms $A\colon\ell_{p_{s_{1}}}^{n}\times\cdots\times
\ell_{p_{s_{k}}}^{n}\rightarrow\mathbb{K}$ and all positive integers $n$, then%
\[
\sup_{\mathbf{i}_{\widehat{S}}}\left(  \sum_{j_{s_{1}}=1}^{n}\left(
\sum_{j_{_{s_{2}}}=1}^{n}\cdots\left(  \sum_{j_{_{s_{k}}}=1}^{n}\left\vert
T(e_{j_{1}},...,e_{j_{m}})\right\vert ^{q_{k}}\right)  ^{\frac{q_{k-1}}{q_{k}%
}}\cdots\right)  ^{\frac{q_{1}}{q_{2}}}\right)  ^{\frac{1}{q_{1}}}\leq
C\left\Vert T\right\Vert
\]
for all $m$-linear forms $T\colon\ell_{p_{1}}^{n}\times\cdots\times\ell
_{p_{m}}^{n}\rightarrow\mathbb{K}$ and all positive integers $n$. Moreover,
if
\[
\left\vert \frac{1}{\mathbf{p}}\right\vert _{S}<1
\]
and, for every $j\in\widehat{S}$,
\[
\left\vert \frac{1}{\mathbf{p}}\right\vert _{S\cup\left\{  j\right\}  }\geq1,
\]
the $\sup$ cannot be improved.
\end{lemma}

\begin{proof}
To simplify the notation, we can suppose $\left(  s_{1},...,s_{k}\right)
=\left(  1,...,k\right)  $.

Let us fix the last $m-k$ variables and work with $k$-linear forms
$S\colon\ell_{p_{1}}^{n}\times\cdots\times\ell_{p_{k}}^{n}\rightarrow
\mathbb{K}$. Since
\[
\left(  \sum_{j_{1}=1}^{n}\left(  \sum_{j_{2}=1}^{n}\cdots\left(  \sum
_{j_{k}=1}^{n}\left\vert A(e_{j_{1}},...,e_{j_{k}})\right\vert ^{q_{k}%
}\right)  ^{\frac{q_{k-1}}{q_{k}}}\cdots\right)  ^{\frac{q_{1}}{q_{2}}%
}\right)  ^{\frac{1}{q_{1}}}\leq C\left\Vert A\right\Vert
\]
for all $k$-linear forms $A\colon\ell_{p_{1}}^{n}\times\cdots\times\ell
_{p_{k}}^{n}\rightarrow\mathbb{K}$, we know that there is a constant $C\geq1$,
such that for any fixed vectors $e_{j_{k+1}},...,e_{j_{m}}$, we have
\begin{align*}
&  \left(  \sum_{j_{s_{1}}=1}^{n}\left(  \sum_{j_{_{s_{2}}}=1}^{n}%
\cdots\left(  \sum_{j_{_{s_{k}}}=1}^{n}\left\vert T(e_{j_{1}},...,e_{j_{m}%
})\right\vert ^{q_{k}}\right)  ^{\frac{q_{k-1}}{q_{k}}}\cdots\right)
^{\frac{q_{1}}{q_{2}}}\right)  ^{\frac{1}{q_{1}}}\\
&  \leq C\left\Vert T\left(  \cdot,\cdots,\cdot,e_{j_{k+1}},...,e_{j_{m}%
}\right)  \right\Vert
\end{align*}
for all $m$-linear forms $T\colon\ell_{p_{1}}^{n}\times\cdots\times\ell
_{p_{m}}^{n}\rightarrow\mathbb{K}$. Then, there is a constant $C\geq1$, such
that
\begin{align*}
&  \sup_{\mathbf{i}_{\widehat{S}}}\left(  \sum_{j_{s_{1}}=1}^{n}\left(
\sum_{j_{_{s_{2}}}=1}^{n}\cdots\left(  \sum_{j_{_{s_{k}}}=1}^{n}\left\vert
T(e_{j_{1}},...,e_{j_{m}})\right\vert ^{q_{k}}\right)  ^{\frac{q_{k-1}}{q_{k}%
}}\cdots\right)  ^{\frac{q_{1}}{q_{2}}}\right)  ^{\frac{1}{q_{1}}}\\
&  \leq C\sup_{\mathbf{i}_{\widehat{S}}}\left\Vert T\left(  \cdot,\cdots
,\cdot,e_{j_{k+1}},...,e_{j_{m}}\right)  \right\Vert \\
&  \leq C\left\Vert T\right\Vert .
\end{align*}
for all $m$-linear forms $T\colon\ell_{p_{1}}^{n}\times\cdots\times\ell
_{p_{m}}^{n}\rightarrow\mathbb{K}$.

Now let us show that the $\sup$ cannot be improved. In fact, in this case we
have $m-k$ suprema and no one can be improved. Otherwise there will exist
$i\in\widehat{S}$, $r\in\left(  0,\infty\right)  $ and $C\geq1$ such that%
\[
\sup_{\mathbf{i}_{\widehat{S\cup\left\{  i\right\}  }}}\left(  \sum_{j_{i}%
=1}^{n}\left(  \sum_{j_{s_{1}}=1}^{n}\left(  \sum_{j_{_{s_{2}}}=1}^{n}%
\cdots\left(  \sum_{j_{_{s_{k}}}=1}^{n}\left\vert T(e_{j_{1}},...,e_{j_{m}%
})\right\vert ^{q_{k}}\right)  ^{\frac{q_{k-1}}{q_{k}}}\cdots\right)
^{\frac{q_{1}}{q_{2}}}\right)  ^{\frac{r}{q_{1}}}\right)  ^{\frac{1}{r}}\leq
C\left\Vert T\right\Vert
\]
for all $m$-linear forms $T\colon\ell_{p_{_{1}}}^{n}\times\cdots\times
\ell_{p_{m}}^{n}\rightarrow\mathbb{K}$ and all $n$. Using the Lemma \ref{t0},
this would imply the existence of a constant $C\geq1$ \ such that
\[
\left(  \sum_{j_{i}=1}^{n}\left(  \sum_{j_{s_{1}}=1}^{n}\left(  \sum
_{j_{_{s_{2}}}=1}^{n}\cdots\left(  \sum_{j_{_{s_{k}}}=1}^{n}\left\vert
A(e_{j_{i}},e_{j_{s_{1}}},...,e_{j_{s_{k}}})\right\vert ^{q_{k}}\right)
^{\frac{q_{k-1}}{q_{k}}}\cdots\right)  ^{\frac{q_{1}}{q_{2}}}\right)
^{\frac{r}{q_{1}}}\right)  ^{\frac{1}{r}}\leq C\left\Vert A\right\Vert
\]
for all $\left(  k+1\right)  $-linear forms $A\colon\ell_{p_{i}}^{n}\times
\ell_{p_{s_{1}}}^{n}\times\cdots\times\ell_{p_{s_{k}}}^{n}\rightarrow
\mathbb{K}$. Considering $\rho=\max\left\{  q_{1},...,q_{k},r\right\}  ,$ by
the monotonicity of the $\ell_{q}~$norms we conclude that there is a constant
$C\geq1$ such that%
\[
\left(  \sum_{j_{i},j_{s_{1}},...,j_{s_{k}}=1}^{n}\left\vert A(e_{j_{i}%
},e_{j_{s_{1}}},...,e_{j_{s_{k}}})\right\vert ^{\rho}\right)  ^{\frac{1}{\rho
}}\leq C\left\Vert A\right\Vert
\]
for all $\left(  k+1\right)  $-linear forms $A\colon\ell_{p_{i}}^{n}\times
\ell_{p_{s_{1}}}^{n}\times\cdots\times\ell_{p_{s_{k}}}^{n}\rightarrow
\mathbb{K}$. But this is impossible due to the hypothesis $\left\vert \frac
{1}{\mathbf{p}}\right\vert _{S\cup\left\{  i\right\}  }\geq1$.
\end{proof}

\begin{lemma}
\label{t0}Let $k\in\left\{  1,...,m\right\}  $ and $\mathbf{p=}\left(
p_{1},...,p_{m}\right)  \in\lbrack1,\infty]^{m}$. If there is a constant
$C\geq1$ such that%
\[
\left(  \sum_{j_{s_{1}}=1}^{n}\left(  \sum_{j_{_{s_{2}}}=1}^{n}\cdots\left(
\sum_{j_{_{s_{m}}}=1}^{n}\left\vert T(e_{j_{s_{1}}},...,e_{j_{s_{m}}%
})\right\vert ^{q_{m}}\right)  ^{\frac{q_{m-1}}{q_{m}}}\cdots\right)
^{\frac{q_{1}}{q_{2}}}\right)  ^{\frac{1}{q_{1}}}\leq C\left\Vert
T\right\Vert
\]
for all $m$-linear forms $T\colon\ell_{p_{s_{1}}}^{n}\times\cdots\times
\ell_{p_{s_{m}}}^{n}\rightarrow\mathbb{K}$ and all positive integers $n$, then%
\[
\left(  \sum_{j_{s_{k+1}}=1}^{n}\left(  \sum_{j_{_{s_{k+2}}}=1}^{n}%
\cdots\left(  \sum_{j_{_{s_{m}}}=1}^{n}\left\vert A(e_{j_{s_{k+1}}%
},...,e_{j_{s_{m}}})\right\vert ^{q_{m}}\right)  ^{\frac{q_{m-1}}{q_{m}}%
}\cdots\right)  ^{\frac{q_{k+1}}{q_{k+2}}}\right)  ^{\frac{1}{q_{k+1}}}\leq
C\left\Vert A\right\Vert
\]
for all $\left(  m-k\right)  $-linear forms $A\colon\ell_{p_{s_{k+1}}}%
^{n}\times\cdots\times\ell_{p_{s_{m}}}^{n}\rightarrow\mathbb{K}$ and all
positive integers $n$.
\end{lemma}

\begin{proof}
To simplify the notation, we can suppose $\left(  s_{1},...,s_{m}\right)
=\left(  1,...,m\right)  $.

Let suppose that there is a constant $C\geq1$ such that%
\[
\left(  \sum_{j_{1}=1}^{n}\left(  \sum_{j_{_{2}}=1}^{n}\cdots\left(
\sum_{j_{_{m}}=1}^{n}\left\vert T(e_{j_{1}},...,e_{j_{m}})\right\vert ^{q_{m}%
}\right)  ^{\frac{q_{m-1}}{q_{m}}}\cdots\right)  ^{\frac{q_{1}}{q_{2}}%
}\right)  ^{\frac{1}{q_{1}}}\leq C\left\Vert T\right\Vert
\]
for all $m$-linear forms $T\colon\ell_{p_{1}}^{n}\times\cdots\times\ell
_{p_{m}}^{n}\rightarrow\mathbb{K\,}$.

Given a $\left(  m-k\right)  $-linear form $S\colon\ell_{p_{k+1}}^{n}%
\times\cdots\times\ell_{p_{m}}^{n}\rightarrow\mathbb{K}$, we define the
$m$-linear form $T\colon\ell_{p_{1}}^{n}\times\cdots\times\ell_{p_{m}}%
^{n}\rightarrow\mathbb{K},$ given by
\[
T\left(  x^{\left(  1\right)  },x^{\left(  2\right)  },\ldots,x^{\left(
m\right)  }\right)  =x_{1}^{\left(  1\right)  }\cdots x_{1}^{\left(  k\right)
}S\left(  x^{\left(  k+1\right)  },x^{\left(  k+2\right)  },\ldots,x^{\left(
m\right)  }\right)  .
\]
It is obvious that $\left\Vert T\right\Vert =\left\Vert S\right\Vert ;$ then,
by the above assumption there is a constant $C\geq1$ such that
\begin{align*}
&  \left(  \sum_{j_{k+1}=1}^{n}\left(  \sum_{j_{_{k+2}}=1}^{n}\cdots\left(
\sum_{j_{_{m}}=1}^{n}\left\vert S(e_{j_{s_{k+1}}},...,e_{j_{s_{m}}%
})\right\vert ^{q_{m}}\right)  ^{\frac{q_{m-1}}{q_{m}}}\cdots\right)
^{\frac{q_{k+1}}{q_{k+2}}}\right)  ^{\frac{1}{q_{k+1}}}\\
&  =\sup_{\mathbf{i}_{\widehat{\left\{  k+1,...,m\right\}  }}}\left(
\sum_{j_{k+1}=1}^{n}\left(  \sum_{j_{_{k+2}}=1}^{n}\cdots\left(  \sum
_{j_{_{m}}=1}^{n}\left\vert e_{1}^{\left(  1\right)  }\cdots e_{1}^{\left(
k\right)  }S(e_{j_{s_{k+1}}},...,e_{j_{s_{m}}})\right\vert ^{q_{m}}\right)
^{\frac{q_{m-1}}{q_{m}}}\cdots\right)  ^{\frac{q_{k+1}}{q_{k+2}}}\right)
^{\frac{1}{q_{k+1}}}\\
&  =\sup_{\mathbf{i}_{\widehat{\left\{  k+1,...,m\right\}  }}}\left(
\sum_{j_{k+1}=1}^{n}\left(  \sum_{j_{_{k+2}}=1}^{n}\cdots\left(  \sum
_{j_{_{m}}=1}^{n}\left\vert T(e_{j_{1}},...,e_{j_{s_{m}}})\right\vert ^{q_{m}%
}\right)  ^{\frac{q_{m-1}}{q_{m}}}\cdots\right)  ^{\frac{q_{k+1}}{q_{k+2}}%
}\right)  ^{\frac{1}{q_{k+1}}}\\
&  \leq\left(  \sum_{j_{1}=1}^{n}\left(  \sum_{j_{_{2}}=1}^{n}\cdots\left(
\sum_{j_{_{m}}=1}^{n}\left\vert T(e_{j_{1}},...,e_{j_{m}})\right\vert ^{q_{m}%
}\right)  ^{\frac{q_{m-1}}{q_{m}}}\cdots\right)  ^{\frac{q_{1}}{q_{2}}%
}\right)  ^{\frac{1}{q_{1}}}\\
&  \leq C\left\Vert T\right\Vert \\
&  =C\left\Vert S\right\Vert .
\end{align*}

\end{proof}

In the next sections, using Lemma \ref{t2} and Lemma \ref{t0}, we obtain the
super-critical versions of the Hardy--Littlewood inequalities presented in the introduction.

A first natural illustration of the usefulness of Lemma \ref{t2} and Lemma
\ref{t0} lead us to an alternate proof of Proposition 6.3 in Pellegrino et al.
2017. In fact, if $q\in(1,\infty]$, it is well known that
\[
\left(  \sum_{j=1}^{n}\left\vert A\left(  e_{j}\right)  \right\vert
^{a}\right)  ^{\frac{1}{a}}\leq\left\Vert A\right\Vert
\]
for all bounded linear forms $A\colon\ell_{q}\rightarrow\mathbb{K}$, if, and
only if, $a\geq\frac{q}{q-1}.$ Thus, for $a,b\in(0,\infty],$ and
$p,q\in(1,\infty]$ such that $\frac{1}{p}+\frac{1}{q}\geq1$, we invoke Lemma
\ref{t2} and Lemma \ref{t0} to obtain:

\begin{proposition}
\label{pqmaior1} (see Proposition 6.3 in Pellegrino et al. 2017) Let
$p,q\in(1,\infty]$ be such that $\frac{1}{p}+\frac{1}{q}\geq1.$ We have
\[
\left(  \sum_{i=1}^{n}\left(  \sum_{j=1}^{n}\left\vert A(e_{i},e_{j}
)\right\vert ^{a}\right)  ^{\frac{b}{a}}\right)  ^{\frac{1}{b}}\leq\left\Vert
A\right\Vert
\]
for all bilinear forms $A\colon\ell_{p}^{n}\times\ell_{q}^{n}\rightarrow
\mathbb{K}$ and all $n$ if, and only if, the exponents $a,b$ satisfy
\[
b=\infty\text{ and }a\geq\frac{q}{q-1}.
\]

\end{proposition}

\begin{figure}[h]
\begin{tikzpicture}
	
	\draw[->] (0,0) -- (7,0) node[below] {$p$}; 
	\draw[->] (0,0) -- (0,7) node[left] {$q $}; 
	
	\draw[domain=7/6:6/5, color=blue, thick, dotted, smooth] plot (\x,{(\x)/(\x-1)}); 
	
	\draw[domain=6/5:6, color=blue, thick, smooth] plot (\x,{(\x)/(\x-1)}); 

	\draw[domain=6:7, color=blue, thick, dotted, smooth] plot (\x,{(\x)/(\x-1)}); 
	
	\draw (5,2) node[below, color=black!99] {$\frac{1}{p}+\frac{1}{q}=1$}; 
	\draw (0,0) node[left] {$0$};
	\draw (1,0) node[below] {$1$};
	\draw (0,1) node[left] {$1$};
	\draw (2,0) node[below] {$2$};
	\draw (0,2) node[left] {$2$};
	
	\draw[dotted] (1,0) -- (1,1);
	\draw[dotted] (0,1) -- (1,1);
	\draw[color=blue] (1,1) -- (1,6);
	
	\draw[dotted] (1,6) -- (1,7);
	
	\draw[color=blue] (1,1) -- (6,1);
	
	\draw[color=green] (2,2) -- (2,6);
	
	\draw[dotted] (2,6) -- (2,7);
	
	\draw[dotted] (2,0) -- (2,7);
	\draw[dotted] (0,2) -- (7,2);
	
	\draw[dotted] (7,1) -- (7,7);
	\draw[dotted] (1,7) -- (7,7);
	\draw[dotted] (6,1) -- (7,1);
	
	
	\path[draw,shade,bottom color=blue!60,top color=blue!9,opacity=.3] plot [smooth,samples=100, domain=7/6:6] (\x,{(\x)/(\x-1)})--plot (6,6/5)--plot(6,1)--plot (1,1) --plot(1,6)--plot (6/5,6);

	\path[draw,shade,bottom color=blue!99,top color=blue!30,opacity=.1] plot(6/5,7)-- plot [smooth,samples=100, domain=7/6:6/5] (\x,{(\x)/(\x-1)})--plot (6/5,6)--plot(1,6)--plot (1,7) --plot(6/5,7);

	\path[draw,shade,bottom color=blue!99,top color=blue!30,opacity=.1] plot(6,6/5)-- plot(6,1)--plot(7,1)--plot(7,7/6)-- plot [smooth,samples=100, domain=6:7] (\x,{(\x)/(\x-1)})--plot (7,7/6);

	\path[draw,shade, bottom color=red!99,top color=red!30,opacity=.3] plot [smooth,samples=100, domain=6/5:2] (\x,{(\x)/(\x-1)})--plot (2,2)--plot(2,6)--plot (6/5,6);

	\path[draw,shade, bottom color=red!99,top color=red!20,opacity=.1] plot [smooth,samples=100, domain=7/6:6/5] (\x,{(\x)/(\x-1)})--plot (6/5,6)--plot(2,6)--plot (2,7);
	
	\path[draw,shade,bottom color=green!99,top color=green!30,opacity=.3] plot [smooth,samples=100, domain=2:6] (\x,{(\x)/(\x-1)})--plot (6,6/5)--plot(6,6)--plot (2,6) --plot(2,2);

	\path [bottom color=green,top color=green,opacity=.05] (2,7) rectangle (7,6);
	
	\path [bottom color=green,top color=green,opacity=.05] (6,6/5) rectangle (7,6);

	\draw (-1,-1) node[right] {Theorem \ref{azx}};
	\path [bottom color=green,top color=green,opacity=.3] (-1,-0.7) rectangle (-1.5, -1.2);

	\draw (2.5,-1) node[right] {Unknown};
	\path [bottom color=red,top color=red,opacity=.3] (2,-0.7) rectangle (2.5, -1.2);

	\draw (5.3,-1) node[right] {Proposition \ref{pqmaior1}};
	\path [bottom color=blue,top color=blue,opacity=.3] (5.3,-0.7) rectangle (4.7, -1.2);

	\end{tikzpicture}
\caption{Classical bilinear case.}%
\label{ddssdds}%
\end{figure}
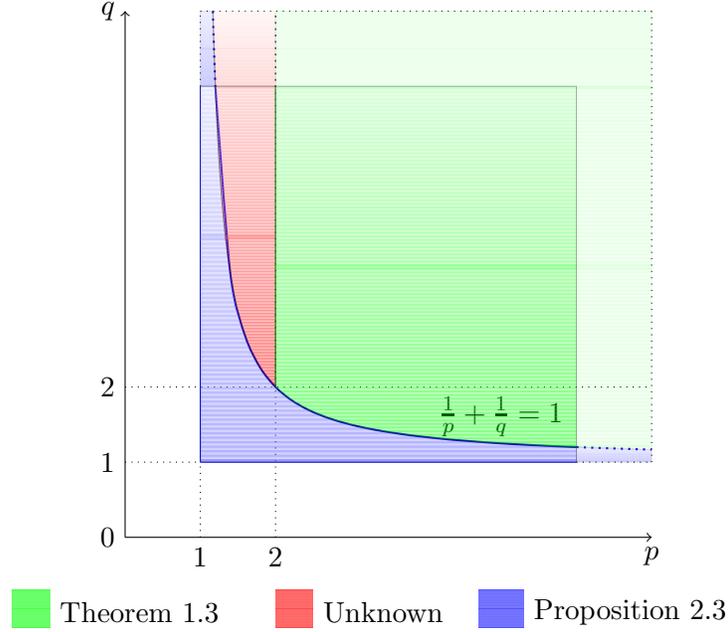



In this section we are mainly interested in the case of $3$-linear forms.

By Theorem \ref{criticodjair} used for $3$-linear forms we have%
\begin{equation}
\sup_{j_{1}}\left(  \sum_{j_{2}=1}^{n}\left(  \sum_{j_{3}=1}^{n}\left\vert
T\left(  e_{j_{1}},e_{j_{2}},e_{j_{3}}\right)  \right\vert ^{q_{3}}\right)
^{\frac{q_{2}}{q_{3}}}\right)  ^{\frac{1}{q_{2}}}\leq\sqrt{2}\left\Vert
T\right\Vert \label{zzs}%
\end{equation}
for all $3$-linear forms $T:\ell_{3}^{n}\times\ell_{3}^{n}\times\ell_{3}%
^{n}\rightarrow\mathbb{K},$ and all positive integers $n$, with $q_{2}=3$ and
$q_{3}=12/5$. Moreover, $q_{1}=\infty$ and $q_{2}=3$ are sharp and the optimal
exponent $q_{3}$ satisfying (\ref{zzs}) fulfill $q_{3}\geq3/2.$

As a consequence of Lemma \ref{t2} and Lemma \ref{t0}, we complete the above result.

\begin{proposition}
\label{444}\ Let $p,r\in\left(  1,\infty\right)  $ and $q\in[2,\infty]$ be
such that $\frac{1}{q}+\frac{1}{r}<1$ and $\frac{1}{p}+\frac{1}{q}+\frac{1}%
{r}\geq1$. The following assertions are equivalent:

\begin{itemize}
\item[(a)] There is a constant $C\geq1$ such that
\[
\left(  \sum_{j_{1}=1}^{n}\left(  \sum_{j_{2}=1}^{n}\left(  \sum_{j_{3}=1}%
^{n}\left\vert T\left(  e_{j_{1}},e_{j_{2}},e_{j_{3}}\right)  \right\vert
^{q_{3}}\right)  ^{\frac{q_{2}}{q_{3}}}\right)  ^{\frac{q_{1}}{q_{2}}}\right)
^{\frac{1}{q_{1}}}\leq C\left\Vert T\right\Vert \label{71}%
\]
for every $3$-linear forms $T:\ell_{p}^{n}\times\ell_{q}^{n}\times\ell_{r}%
^{n}\rightarrow\mathbb{K}$ and all $n$.

\item[(b)] The exponents $a,b,c$ satisfy
\[
q_{1}=\infty,~q_{2}\geq\frac{1}{1-\left(  \frac{1}{r}+\frac{1}{q}\right)
},\text{ }q_{3}\geq\frac{r}{r-1},
\]
and
\[
\frac{1}{q_{2}}+\frac{1}{q_{3}}\leq\frac{3}{2}-\left(  \frac{1}{r}+\frac{1}%
{q}\right)  .
\]

\end{itemize}
\end{proposition}

\begin{proof}
Since $\frac{1}{q}+\frac{1}{r}<1$, by Theorem \ref{azx} there is a constant
$C\geq1$ such that%
\[
\left(  \sum_{j_{2}=1}^{n}\left(  \sum_{j_{3}=1}^{n}\left\vert A\left(
e_{j_{2}},e_{j_{3}}\right)  \right\vert ^{q_{3}}\right)  ^{\frac{q_{2}}{q_{3}%
}}\right)  ^{\frac{1}{q_{2}}}\leq C\left\Vert A\right\Vert
\]
for all bilinear forms $A\colon\ell_{q}^{n}\times\ell_{r}^{n}\rightarrow
\mathbb{K}$ if, and only if,%
\[
q_{3}\geq\frac{r}{r-1},q_{2}\geq\frac{1}{1-\left(  \frac{1}{r}+\frac{1}%
{q}\right)  }.
\]
and
\[
\frac{1}{q_{3}}+\frac{1}{q_{2}}\leq\frac{3}{2}-\left(  \frac{1}{r}+\frac{1}%
{q}\right)  .
\]
We combine this equivalence with the fact $\frac{1}{q}+\frac{1}{r}<1$ and
$\frac{1}{p}+\frac{1}{q}+\frac{1}{r}\geq1$, and then, we invoke Lemma \ref{t0}
and Lemma \ref{t2} to conclude the proof.
\end{proof}

\begin{corollary}
\label{trilinear} For all $3$-linear forms $T:\ell_{3}^{n}\times\ell_{3}%
^{n}\times\ell_{3}^{n}\rightarrow\mathbb{K}$ and all $n$, we have
\[
\left(  \sum_{j_{1}=1}^{n}\left(  \sum_{j_{2}=1}^{n}\left(  \sum_{j_{3}=1}%
^{n}\left\vert T\left(  e_{j_{1}},e_{j_{2}},e_{j_{3}}\right)  \right\vert
^{q_{3}}\right)  ^{\frac{q_{2}}{q_{3}}}\right)  ^{\frac{q_{1}}{q_{2}}}\right)
^{\frac{1}{q_{1}}}\leq C\left\Vert T\right\Vert
\]
if, and only if, $q_{1}=\infty,$ $q_{2}\geq3,~q_{3}\geq3/2,$ and $\frac
{1}{q_{2}}+\frac{1}{q_{3}}\leq\frac{5}{6}$.
\end{corollary}

\begin{figure}[tbh]
\begin{tikzpicture}
	
	\draw[->] (1,1) -- (8,1) node[below] {$q_2$}; 
	\draw[->] (1,1) -- (1,8) node[left] {$q_3$}; 
	
	\draw[domain=3:6, color=blue, thick, smooth] plot (\x,{(6*\x)/(5*\x-6)}); 

	\draw (4.3,2.6) node[below, color=black!99] {$\frac{1}{q_2}+\frac{1}{q_3}=\frac{5}{6}$}; 
	\draw (1,1) node[left] {$1$};
	\draw (1,1) node[below] {$1$};
	\draw (1,2) node[left] {$2$};
	\draw (2,1) node[below] {$2$};
	\draw (1,3/2) node[left] {$3/2$};
		\draw (3,1) node[below] {$3$};
			\draw (6,1) node[below] {$6$};	
	
	\draw[dotted] (1,3/2) -- (6,3/2);
	\draw[dotted] (3,1) -- (3,8);
	
\draw [color=blue] (3,2) -- (3,7);
\draw [color=blue] (6,3/2) -- (7,3/2);
	\draw[dotted] (7,3/2) -- (8,3/2);
	
		\draw[dotted] (0,2) -- (3,2);
				\draw[dotted] (6,1) -- (6,3/2);
				

	\path[ draw,bottom color=blue!99,top color=blue!30,opacity=.3] plot [smooth,samples=100, domain=3:6] (\x,{(6*\x)/(5*\x-6)})--plot(7,3/2)--plot(7,7)--plot(3,7);
	
	\path[ draw,bottom color=blue!99,top color=blue!30,opacity=.1] plot(7,3/2) --plot(8,3/2)--plot(8,8)--plot(3,8)--plot(3,7)--plot(7,7);

	\draw (1.3,-0.1) node[right] {Admissible exponents for the critical 3-linear case};
	\path [bottom color=blue,top color=blue,opacity=.3] (1.3,-0.3) rectangle (0.7, 0.2);
	
	\end{tikzpicture}
\caption{Critical 3-linear case.}%
\label{figuratri}%
\end{figure}

\section{The $m$-linear case}

Now we use Lemma \ref{t2} and Lemma \ref{t0} to obtain super-critical versions
of Hardy--Littlewood inequalities for $m$-linear forms. Our main result is the
following Theorem. Below, we use the notation $\lceil x\rceil$ to represent
the smallest integer greater than to $x$, i.e., $\lceil x\rceil=\min
\,\{n\in\mathbb{Z}\mid n>x\}.$

\begin{theorem}
\label{dimantgeral} Let $m\geq2$ be a positive integer, $p\in(1,2m]$,
$k:=\max\{0,\lceil m-p\rceil\}$ and $A=\{i\in\{1,...,m-1\}: i\leq k\}$. Then,
there is a constant $C\geq1$ such that
\[
\sup_{j_{i}, i \in A}\left(  \sum_{j_{k+1},...j_{m}=1}^{n}\left\vert
T(e_{j_{1}},...,e_{j_{m}})\right\vert ^{q}\right)  ^{\frac{1}{q}}\leq
C\left\Vert T\right\Vert
\]
for every $m$-linear forms $T:\ell_{p}^{n}\times\cdots\times\ell_{p}
^{n}\rightarrow\mathbb{K}$ if, and only if,
\[
q\geq\frac{p}{p-\left(  m-k\right)  }.
\]
Moreover, the $\sup$ cannot be improved.
\end{theorem}

\begin{proof}
The case $k=0$ is precisely (\ref{bvf}), so we shall assume $k\geq1$. Since
$p\in(m-k,m-k+1]$ we have
\[
\frac{1}{m-k+1}\leq\frac{1}{p}<\frac{1}{m-k}%
\]
and thus
\[
\frac{m-k}{m-k+1}\leq\frac{m-k}{p}<1.
\]
On the other hand we also have
\[
1\leq\frac{m-k+1}{p}.
\]
By (\ref{bvf}) there is a constant $C\geq1$ such that
\[
\left(  \sum_{j_{k+1},...j_{m}=1}^{n}\left\vert T(e_{j_{k+1}},...,e_{j_{k}%
})\right\vert ^{q}\right)  ^{\frac{1}{q}}\leq C\left\Vert T\right\Vert
\]
for every $\left(  m-k\right)  $-linear forms $T:\ell_{p}^{n}\times
\cdots\times\ell_{p}^{n}\rightarrow\mathbb{K}$ if, and only if,%
\[
q\geq\frac{p}{p-\left(  m-k\right)  }.
\]
By Lemma \ref{t2} with $S=\{k+1,k+2,\dots,m\}\subset\{1,\dots,m\}$, and Lemma
\ref{t0} we conclude the proof.
\end{proof}

We finish this section with some super-critical results in the anisotropic
setting, whose proofs we omit. We begin with a super-critical version of
Theorem \ref{aron}:

\begin{theorem}
Let $m\geq2$, $k\in\left\{  1,...,m-1\right\}  $, $p_{1},...,p_{k}\in
[1,\infty]$, $p_{k+1},...,p_{m-1}\in(2,\infty]$ and $p_{m}\in(1,2],$ such
that
\[
\frac{1}{p_{k+1}}+\cdots+\frac{1}{p_{m}}<1
\]
and
\[
\frac{1}{p_{j}}+\frac{1}{p_{k+1}}+\cdots+\frac{1}{p_{m}}\geq1
\]
for all $j\in\{1,...,k\}.$ The following assertions are equivalent:

\begin{itemize}
\item[$(a)$] There is a constant $C\geq1$ such that
\[
\left(  \sum_{j_{1}=1}^{n}\left(  \sum_{j_{2}=1}^{n}\cdots\left(  \sum
_{j_{m}=1}^{n}\left\vert T(e_{j_{1}},...,e_{j_{m}})\right\vert ^{q_{m}%
}\right)  ^{\frac{q_{m-1}}{q_{m}}}\cdots\right)  ^{\frac{q_{1}}{q_{2}}%
}\right)  ^{\frac{1}{q_{1}}}\leq C\left\Vert A\right\Vert
\]
for all $m$-linear forms $T:\ell_{p}^{n}\times\cdots\times\ell_{p}%
^{n}\rightarrow\mathbb{K}$ and all $n.$

\item[$(b)$] The exponents satisfy
\[
q_{1}=\cdots=q_{k}=\infty~\mbox{and}~ q_{i}\geq\dfrac{1}{1-\left(  \frac
{1}{p_{i}}+\cdots+\frac{1}{p_{m}}\right)  }, i=k+1,...,m.
\]

\end{itemize}
\end{theorem}

Analogously, using Lemma \ref{t2}, Lemma \ref{t0} and Theorem \ref{abps} we have:

\begin{theorem}
\label{t4} Let $p_{1},...,p_{k}\in\lbrack1,2]$ and $p_{k+1},...,p_{m}%
\in\lbrack2,\infty]$ be such that
\[
\frac{1}{p_{k+1}}+\cdots+\frac{1}{p_{m}}\leq\frac{1}{2}%
\]
and%
\[
\frac{1}{p_{j}}+\frac{1}{p_{k+1}}+\cdots+\frac{1}{p_{m}}\geq1
\]
for all $j\in\left\{  1,...,k\right\}  $, and%
\[
q_{k+1},...,q_{m}\in\left[  \frac{1}{1-\left(  \frac{1}{p_{k+1}}+\cdots
+\frac{1}{p_{m}}\right)  },2\right]  .
\]
The following assertions are equivalent:

(a) There is a constant $C$ (not depending on $n)$ such that
\[
\left(  \sum_{j_{1}=1}^{n}\left(  \cdots\left(  \sum_{j_{m}=1}^{n}\left\vert
T\left(  e_{j_{1}},\dots,e_{j_{m}}\right)  \right\vert ^{q_{m}}\right)
^{\frac{q_{m-1}}{q_{m}}}\cdots\right)  ^{\frac{q_{1}}{q_{2}}}\right)
^{\frac{1}{q_{1}}}\leq C\left\Vert T\right\Vert ,
\]
for all $m$-linear forms $T:\ell_{p_{1}}^{n}\times\cdots\times\ell_{p_{m}}%
^{n}\longrightarrow\mathbb{K}$ and all positive integers $n$.

(b) $q_{1}=\cdots=q_{k}=\infty$ and the inequality
\[
\frac{1}{q_{k+1}}+\cdots+\frac{1}{q_{m}}\leq\dfrac{(m-k)+1}{2}-\left(
\frac{1}{p_{k+1}}+\cdots+\frac{1}{p_{m}}\right)
\]
is verified.
\end{theorem}

The next result shows that it is possible to avoid the condition $\frac
{1}{p_{j}}+\frac{1}{p_{k+1}}+\cdots+\frac{1}{p_{m}}\geq1$, for all
$j\in\left\{  1,...,k\right\}  $:

\begin{theorem}
\label{t5} Let $p_{1},...,p_{k}\in\lbrack1,2]$ and $p_{k+1},...,p_{m}%
\in\lbrack2,\infty]$ be such that
\[
\frac{1}{p_{k+1}}+\cdots+\frac{1}{p_{m}}\leq\frac{1}{2}%
\]
and%
\[
q_{k+1},...,q_{m}\in\left[  \frac{1}{1-\left(  \frac{1}{p_{k+1}}+\cdots
+\frac{1}{p_{m}}\right)  },2\right]
\]
with
\[
\frac{1}{q_{k+1}}+\cdots+\frac{1}{q_{m}}=\dfrac{(m-k)+1}{2}-\left(  \frac
{1}{p_{k+1}}+\cdots+\frac{1}{p_{m}}\right)  .
\]
The following assertions are equivalent:

(a) There is a constant $C$ (not depending on $n)$ such that
\[
\left(  \sum_{j_{1}=1}^{n}\left(  \cdots\left(  \sum_{j_{m}=1}^{n}\left\vert
T\left(  e_{j_{1}},\dots,e_{j_{m}}\right)  \right\vert ^{q_{m}}\right)
^{\frac{q_{m-1}}{q_{m}}}\cdots\right)  ^{\frac{q_{1}}{q_{2}}}\right)
^{\frac{1}{q_{1}}}\leq C\left\Vert T\right\Vert ,
\]
for all $m$-linear forms $T:\ell_{p_{1}}^{n}\times\cdots\times\ell_{p_{m}}%
^{n}\longrightarrow\mathbb{K}$ and all positive integers $n$.

(b) $q_{1}=\cdots=q_{k}=\infty$.
\end{theorem}

\begin{proof}
Suppose that $(a)$ holds and $q_{k}<\infty$. In this case, Lemma \ref{t0}
provides a constant $C$ such that
\[
\left(  \sum_{j_{k}=1}^{n}\left(  \cdots\left(  \sum_{j_{m}=1}^{n}\left\vert
T\left(  e_{j_{1}},\dots,e_{j_{m}}\right)  \right\vert ^{q_{m}}\right)
^{\frac{q_{m-1}}{q_{m}}}\cdots\right)  ^{\frac{q_{k}}{q_{k+1}}}\right)
^{\frac{1}{q_{k}}}\leq C\left\Vert T\right\Vert
\]
for all $\left(  m-k+1\right)  $-linear forms $T:\ell_{p_{k}}^{n}\times
\cdots\times\ell_{p_{m}}^{n}\rightarrow\mathbb{K}$ and all positive integers
$n$. For any $\left(  m-k+1\right)  $-linear form $T:\ell_{p_{k}}^{n}%
\times\cdots\times\ell_{p_{m}}^{n}\rightarrow\mathbb{K}$, we define a $\left(
m-k+1\right)  $-linear form $S$ with the same rule of $T$, but different
domain $\ell_{2}^{n}\times\ell_{p_{k+1}}^{n}\times\cdots\times\ell_{p_{m}}%
^{n}$. So, there is a constant $C$ such that
\begin{align}
\label{2357} &  \left(  \sum_{j_{k}=1}^{n}\left(  \cdots\left(  \sum_{j_{m}%
=1}^{n}\left\vert S\left(  e_{j_{1}},\dots,e_{j_{m}}\right)  \right\vert
^{q_{m}}\right)  ^{\frac{q_{m-1}}{q_{m}}}\cdots\right)  ^{\frac{q_{k}}%
{q_{k+1}}}\right)  ^{\frac{1}{q_{k}}}\\
&  =\left(  \sum_{j_{k}=1}^{n}\left(  \cdots\left(  \sum_{j_{m}=1}%
^{n}\left\vert T\left(  e_{j_{1}},\dots,e_{j_{m}}\right)  \right\vert ^{q_{m}%
}\right)  ^{\frac{q_{m-1}}{q_{m}}}\cdots\right)  ^{\frac{q_{k}}{q_{k+1}}%
}\right)  ^{\frac{1}{q_{k}}}\nonumber\\
&  \leq C\left\Vert T\right\Vert \nonumber\\
&  \leq C\left\Vert S\right\Vert .\nonumber
\end{align}
for all $\left(  m-k+1\right)  $-linear forms $S:\ell_{2}^{n}\times
\ell_{p_{k+1}}^{n}\times\cdots\times\ell_{p_{m}}^{n}\rightarrow\mathbb{K}$,
and the exponents satisfy
\begin{align*}
\frac{1}{q_{k}}+\frac{1}{q_{k+1}}+\cdots+\frac{1}{q_{m}}  &  =\frac{1}{q_{k}%
}+\dfrac{(m-k)+1}{2}-\left(  \frac{1}{p_{k+1}}+\cdots+\frac{1}{p_{m}}\right)
\\
&  >\dfrac{(m-k)+1}{2}-\left(  \frac{1}{p_{k+1}}+\cdots+\frac{1}{p_{m}}\right)
\\
&  =\dfrac{(m-k+1)+1}{2}-\left(  \frac{1}{2}+\frac{1}{p_{k+1}}+\cdots+\frac
{1}{p_{m}}\right)  .
\end{align*}

On the other hand, replacing the unimodular $(m-k+1)$-linear form of the
Kahane--Salem--Zygmund inequality (see Lemma 6.1 in Albuquerque et al. 2014)
in (\ref{2357}), we obtain
\[
n^{\frac{1}{q_{k}}+\cdots+\frac{1}{q_{m}}}\leq C_{m}\cdot n^{\frac
{(m-k+1)+1}{2}-\left(  \frac{1}{2}+\frac{1}{p_{k+1}}+\cdots+\frac{1}{p_{m}%
}\right)  .}%
\]
Since this is valid for all $n$, we conclude that
\[
\frac{1}{q_{k}}+\cdots+\frac{1}{q_{m}}\leq\frac{(m-k+1)+1}{2}-\left(  \frac
{1}{2}+\frac{1}{p_{k+1}}+\cdots+\frac{1}{p_{m}}\right)  ,
\]
and this is a contradiction. Hence $q_{k}=\infty$. Finally, the fact that
$q_{1}=\cdots=q_{k-1}=\infty$ is a consequence of Lemma \ref{t2}, because
\[
\frac{1}{p_{j}}+\frac{1}{p_{k}}+\frac{1}{p_{k+1}}+\cdots+\frac{1}{p_{m}}\geq1
\]
for all $j\in\left\{  1,...,.k-1\right\}  $ (recall that $p_{1},...,p_{k}%
\in\left[  1,2\right]  $).

Finally, using Theorem \ref{abps} and Lemma \ref{t2} we prove that $(b)$
implies $(a)$.
\end{proof}

\begin{remark}
It is worth mentioning that the above theorems are independent. For instance,
if $m=4$, $k=2$, $p_{1}=p_{2}=2$ and $p_{3}=p_{4}=8$, nothing can be inferred
by Theorem \ref{t4}. However, using Theorem \ref{t5}, we conclude that if
$q_{3},q_{4}\in[4/3,2]$ and $\frac{1}{q_{3}}+\frac{1}{q_{4}}=\frac{5}{4}$ then
there is a constant $C$ (not depending on $n)$ such that
\[
\left(  \sum_{j_{1}=1}^{n}\left(  \cdots\left(  \sum_{j_{4}=1}^{n}\left\vert
T\left(  e_{j_{1}},\dots,e_{j_{4}}\right)  \right\vert ^{q_{4}}\right)
^{\frac{q_{3}}{q_{4}}}\cdots\right)  ^{\frac{q_{1}}{q_{2}}}\right)  ^{\frac
{1}{q_{1}}}\leq C\left\Vert T\right\Vert ,
\]
for all $4$-linear forms $T:\ell_{2}^{n}\times\ell_{2}^{n}\times\ell_{8}%
^{n}\times\ell_{8}^{n}\longrightarrow\mathbb{K}$ and all positive integers $n$
if, and only if, $q_{1}=q_{2}=\infty$.
\end{remark}

The following result was proved in Albuquerque \& Rezende 2018 (in Corollary 2):

\begin{theorem}
\label{teo111} (see Corollary 2 in Albuquerque $\&$ Rezende 2018) Let $m$ be a
positive integer and $p_{1},...,p_{m}\in\lbrack1,2m]$ and $\dfrac{1}{p_{1}%
}+\cdots+\dfrac{1}{p_{m}}<1$. Then, there is a constant $C$ (not depending on
$n$) such that
\[
\label{1234}\left(  \sum_{j_{1}=1}^{n}\left(  \cdots\left(  \sum_{j_{m}=1}%
^{n}\left\vert T\left(  e_{j_{1}},\dots,e_{j_{m}}\right)  \right\vert ^{q_{m}%
}\right)  ^{\frac{q_{m-1}}{q_{m}}}\cdots\right)  ^{\frac{q_{1}}{q_{2}}%
}\right)  ^{\frac{1}{q_{1}}}\leq C\left\Vert T\right\Vert
\]
for all $m$-linear forms $T:\ell_{p_{1}}^{n}\times\cdots\times\ell_{p_{m}}%
^{n}\rightarrow\mathbb{K}$ and all positive integers $n,$ with
\[
\dfrac{1}{q_{i}}=\frac{1}{2}+\dfrac{(m-i+1)}{2m}-\left(  \frac{1}{p_{i}%
}+\cdots+\frac{1}{p_{m}}\right)  ,
\]
for all $i=1,....,m.$
\end{theorem}

Again, Lemma \ref{t2} and Lemma \ref{t0} combined with the
Kahane--Salem--Zygmund inequality (see Lemma 6.1 in Albuquerque et al. 2014)
and Lemma 3.1 in Aron et al. 2017 give us the following super-critical version
of the Theorem \ref{teo111}:

\begin{theorem}
Let $m\geq2$, $k\in\{1,...,m-1\}$, $p_{1},...,p_{k}\in[1,\infty]$ and
$p_{k+1},...,p_{m}\in(1,2(m-k)]$, be such that%

\[
\frac{1}{p_{k+1}}+\cdots+\frac{1}{p_{m}}<1
\]
and
\[
\frac{1}{p_{j}}+\frac{1}{p_{k+1}}+\cdots+\frac{1}{p_{m}}\geq1,
\]
for all $j\in\{1,...,k\}$. Then
\begin{equation}
\left(  \sum_{j_{1}=1}^{n}\left(  \cdots\left(  \sum_{j_{m}=1}^{n}\left\vert
T\left(  e_{j_{1}},\dots,e_{j_{m}}\right)  \right\vert ^{q_{m}}\right)
^{\frac{q_{m-1}}{q_{m}}}\cdots\right)  ^{\frac{q_{1}}{q_{2}}}\right)
^{\frac{1}{q_{1}}}\leq2^{\frac{m-k-1}{2}}\left\Vert T\right\Vert \label{48}%
\end{equation}
for all $m$-linear forms $T:\ell_{p_{1}}^{n}\times\cdots\times\ell_{p_{m}}%
^{n}\rightarrow\mathbb{K}$ and all positive integers $n$, with $q_{1}%
=\cdots=q_{k}=\infty$ and
\[
\dfrac{1}{q_{i}}=\frac{1}{2}+\dfrac{(m-i+1)}{2(m-k)}-\left(  \dfrac{1}{p_{i}%
}+\cdots+\dfrac{1}{p_{m}}\right)  ,
\]
for all $i=k+1,....,m.$ Moreover, $q_{1}=\cdots=q_{k}=\infty$, and the optimal
exponents $q_{i}$ satisfying (\ref{48}) are such that
\[
q_{i}\geq\frac{1}{1-\left(  \frac{1}{p_{i}}+\cdots+\frac{1}{p_{m}}\right)
},i=k+1,...,m,
\]
and the inequality
\[
\frac{1}{q_{k+1}}+\cdots+\frac{1}{q_{m}}\leq\dfrac{(m-k)+1}{2}-\left(
\frac{1}{p_{k+1}}+\cdots+\frac{1}{p_{m}}\right)
\]
is verified.
\end{theorem}

\begin{remark}
When $k=1$ and $p_{1}=\cdots=p_{m}=m$ we recover Theorem \ref{criticodjair}.
\end{remark}

\end{document}